\documentclass{amsart}
\usepackage{amsthm}
\usepackage{amssymb,amsmath,mathrsfs,amsthm,xspace}
\usepackage[all,2cell]{xy}

\newcounter{zlist}

\newcounter{blist}

\newcounter{rlist}

\swapnumbers
\newtheorem{theorem}{THEOREM}[section]

\newtheorem{proposition}[theorem]{PROPOSITION}

\newtheorem{corollary}[theorem]{COROLLARY}

\newtheorem{example}[theorem]{EXAMPLE}
\numberwithin{equation}{section}

\newcommand{\A}{{\textbf{A}}}
\newcommand{\B}{{\textbf{B}}}

\newcommand{\V}{{\mathscr {V}}}
\newcommand{\Bb}{{{{\textsf{Ban}}}\!^1}}
\newcommand{\ot}{{\widehat{\otimes}}}
\begin{document}

\title{Effective descent morphisms for Banach modules}
 \author{Bachuki Mesablishvili}
\thanks{The work was partially supported by the
Shota Rustaveli National Science Foundation Grants DI/18/5-113/13 and FR/189/5-113/14.}

\begin{abstract}It is proved that a norm-decreasing homomorphism of commutative Banach algebras is an
effective descent morphism for Banach modules  if and only if it is a weak retract.
\end{abstract}

\smallskip


\maketitle

\section{ INTRODUCTION} The present note is a continuation of the previous works on the problem of describing
effective descent morphisms in various monoidal categories \cite{Me1}, \cite{Me2}, \cite{Me3}, \cite{Me4},
\cite{Me5} and aims to study the descent problem for the symmetric monoidal category of Banach spaces (with
linear contractions as morphisms, and the projective tensor product). Recall that Grothendieck's descent theory
for modules in a symmetric monoidal category $\V=(\V, \otimes, I)$ is the study of which morphisms $\iota:  \A \to\mathbf{ B}$
of  $\V$-monoids are effective descent morphisms for modules in the sense that the corresponding extension-of-scalars functor
$B \otimes_A - : {_{\A}\!\V} \to {_{\B}\!\V}$ from the category of (left) $\A$-modules to the category of (left)  $\B$-modules
is comonadic. In this note we prove that effective descent morphisms for Banach modules
are precisely those norm-decreasing homomorphisms of commutative Banach algebra which are weak retracts.

As background to the subject, we refer to S. MacLane \cite{Mc} for
generalities on category theory, to \cite{CLM} and \cite{P} for terminology and general
results on Banach spaces and to G. Janelidze and W. Tholen \cite{JT1}, \cite{JT2} and \cite{JT3} for descent theory.

\section{PRELIMINARIES}

Suppose that $\V$ is a fixed symmetric monoidal closed
category with tensor product $\otimes$, unit object $I$, and internal-hom $[-,-]$.
Recall (\cite{Mc}) that a \emph{monoid} $\A$ in  $\V$ (or $\V$-monoid) consists of an object $A$ of $\V$ endowed with a multiplication $m_A:
A \otimes A \to A$  and unit morphism $e_A : I \to A$ such that the usual identity
and associative conditions are satisfied. A monoid is called \emph{commutative} if the multiplication map is
unchanged when composed with the symmetry.

Recall further that, for any $\V$-monoid $\mathbf{A}=(A, e_\A, m_\A)$, a \emph{left} $\mathbf{A}$-\emph{module} is
a pair $(V, \rho_V)$, where $V$ is an object of $\V$ and $\rho_V: A \otimes V \to V$ is a morphism in $\V$, called
the \emph{action} (or the $\mathbf{A}$-\emph{action}) on $V$, such that
$\rho_V(m_\A \otimes V)=\rho_V(A \otimes \rho_V)$ and $\rho_V (e_\A \otimes V)=1$. For a given $\V$-monoid $\mathbf{A}$,
the left $\mathbf{A}$-modules are the objects of a category ${_{\mathbf{A}}\!\V}$.
A morphism $f:(V,\rho_V)\to (W,\rho_W)$ is a morphism $f: V \to W$ in $\V$ such that $\rho_W (A \otimes f)=f\rho_V$.
Analogously, one has the category $\V_\A$ of right $\A$-modules.

If $\V$ admits coequalizers, then each morphism
$\iota : \textbf{A} \to \textbf{ B}$ of $\V$-monoids gives rise to two functors:
\begin{itemize}
  \item the \emph{restriction--of--scalars functor} $\iota_!: {_{\mathbf{B}}\!\V} \to {_{\mathbf{A}}\!\V},$ where
  for any (left) $\textbf{B}$-module $(V, \varrho_V)$,  $\iota_!(V, \varrho_V)$ is a (left) $\textbf{A}$-module via the action
$A \otimes V \xrightarrow{\iota \ot V} B \ot V \xrightarrow{\varrho_V} V$;
  \item the \emph{extension--of--scalars functor} $B \otimes_\A -: {_{\mathbf{A}}\!\V} \to {_{\mathbf{B}}\!\V},$ where for any
  (left) $\textbf{A}$-module $(W, \rho_W)$, $B \otimes_\A W$ is a (left) $\textbf{B}$-module
  via the action $$B \otimes B \otimes_\A W \xrightarrow{m_B \otimes_\A W} B \otimes_\A W.$$
\end{itemize}
It is well-known that the restriction--of--scalars functor is right adjoint to the extension--of--scalars functor.
$\iota : \textbf{A} \to \textbf{ B}$ is called an \emph{effective descent morphism (for modules)} if
the extension--of--scalars functor $B \otimes_\A -: {_{\A}\!\V} \to {_{\B}\!\V}$ is comonadic.

We henceforth suppose that $\V$ is a symmetric monoidal closed category with equalizers
and coequalizers.

An inspection of the proof of \cite[Theorem 3.7]{Me4} shows that the theorem holds true also for morphisms $\iota :\A \to \B$ of
$\V$-monoids which are central in the sense that the diagram
$$ \xymatrix{ A\otimes B \ar[d]_{\iota \otimes B}\ar[r]^{\tau_{A,B}} & B \otimes A
\ar[r]^{B \otimes \iota} & B \otimes B \ar[d]^{m_\B}\\
B \otimes B \ar[rr]_{m_\B}&& B\,,}
$$ where $\tau$ is the symmetry of the monoidal category, commutes (see, \cite{Me3}).
Hence we can improve \cite[Theorem 3.7]{Me4} slightly as follows.  (Recall that a \emph{regular injective object} in a category is an
object which has the extension property with respect to regular monomorphisms.)

\begin{theorem}\label{main.1} \textit{Let $\V$ have a regular injective object $Q$ such that
the functor
$$[-, Q] \colon \V \to \V^{op}$$ is comonadic,
and let $\iota \colon \mathbf A \to \mathbf B$ be a central morphism of
monoids in $\V$. The following are equivalent:
\begin{itemize}
\item [(i)] $\iota \colon  \mathbf A \to  \mathbf B$ is an effective descent morphism;
\item[(ii)] $\iota \colon  \mathbf A \to  \mathbf B$ is a pure morphism
in ${_{\mathbf{A}}\!\V}$; that is, for any $\mathbf A$-module $V$,
the morphism $$\iota \otimes_{A} V  \colon V =A \otimes_\A V \to B
\otimes_\A V$$ is a regular monomorphism;
\item[(iii)] the morphism $[\iota, Q]:[B, Q]\to [A, Q]$ is a split epimorphism in ${_{\mathbf A}\!\V}$;
\end{itemize}}
\end{theorem}

Note that, if $\iota$  satisfies  any  (and  hence  all)  of  the  above  equivalent  conditions, then
it is a monomorphism and the centrality then implies that $\A$ is commutative.

\section{THE MAIN RESULT}

Let $\mathbb{K}$ denote either the field of real numbers $\mathbb{R}$ or the field of complex numbers $\mathbb{C}$.
Write $\Bb$ for the category whose objects are Banach spaces over $\mathbb{K}$ and whose morphisms are linear contractions.
It is well-known (e.g, see \cite{H}, \cite{Pa}) that $\Bb$ is a symmetric monoidal category with tensor product of two Banach spaces being their \emph{projective tensor product} $\ot$ (see \cite{CLM}) and the unit for this tensor product being $\mathbb{K}$. Moreover, there is a bifunctor
$[-,-]:(\Bb)^{op} \times \Bb \to \Bb$ (the \emph{internal Hom}) making the category $\Bb$ into a
symmetric closed monoidal category. For two Banach spaces $V$ and $W$, $[V,W]$
is the Banach space whose elements are the bounded linear transformations $V \to W$ quipped with the operator
norm. In $\Bb$ all small limits and all small colimits exist (e.g. \cite{B}).

Recall (for example, from \cite{H},  \cite{Pa}) that  (commutative) unital Banach algebras are exactly (commutative)
monoids in the symmetric monoidal category $\Bb$, and that, for any unital Banach algebra $\A$, an object of ${_{\A}\!}\Bb$
is a \emph{(left)} $\Bb$-\emph{module over} $\A$, that is, a Banach space $V$ together with $\Bb$-morphism
$$A\ot V \to V, \,\,\, a \ot v \to av$$ such that
$a(bv)=(ab)v\,\,\,\text{and}\,\,\, e_\A v=v \,\, (a,b \in A, \, v \in V).$
Since the action is a morphism in $\Bb$, the map $(a,v) \to av $ is bilinear and satisfies the condition $\|av\|_V\leq \|a\|_{_\A} \cdot\|v\|_V.$
The morphisms in ${_{\A}\!}\Bb$ are morphisms in $\Bb$ which are $\A$-linear.

If $\A$ is a unital Banach algebra and $V$ is a $\A$-module, then the dual space $V^*=[V, \mathbb{K}]$ of $V$ has
the structure of a Banach $\A$-module, where the action $A \ot [V, \mathbb{K}] \to [V, \mathbb{K}]$
is given by $$a\ot f \longmapsto (v \to f(av)).$$ Moreover, for any morphism $f:V \to W$ in ${_{\A}\!}\Bb$,
the map $f^*:W^*\to V^*$ is again a morphism in ${_{\A}\!}\Bb$.
And one says that $f$ is a \emph{weak  retract} if $f^*$ is a split epimorphism in ${_{\A}\!\Bb}$.

\bigskip

The main result of this note is the following theorem.

\begin{theorem} \label{main.2} Let \emph{$\A$} be a commutative unital Banach algebra, and \emph{ $\iota \colon \A \to \B$}
a norm-decreasing central homomorphism of unital Banach algebras. Then the following conditions are equivalent:
\begin{itemize}
\item [(i)] $\iota$ is an effective descent morphism for Banach modules; that is, the extension--of--scalars functor
\emph{$B \ot_\A -: {_{\A}\!\Bb} \to{_{\B}\!\Bb}$} is comonadic;
\item[(ii)] $\iota $ is a $\ot$--pure morphism in \emph{${{_\A}\!\Bb}$}; that is, for any Banach \emph{$\A$}-module \emph{$V$}\!,
the morphism \emph{$\iota \, \ot_{\A} V  \colon V =A \ot_\A V \to B\ot_\A V$} is an isometric inclusion;
\item[(iii)] $\iota$ is a weak retract in \emph{${_{\A}\!\Bb}$}.
\end{itemize}
\end{theorem}
\begin{proof} Since the regular monomorphisms in $\Bb$ are precisely the isometric inclusions (e.g. \cite[4.3.10.e]{B}),
$\mathbb{K}$ is regular injective in $\Bb$  by  the  Hahn-Banach  Theorem. Moreover,
the functor $[-, \mathbb{K}] \colon (\Bb)^{op} \to \Bb$ is monadic by \cite{L}. Hence
$[-, \mathbb{K}]$, seen as a functor $[-, \mathbb{K}] \colon \Bb \to (\Bb)^{op}$, is comonadic.
One now concludes the proof by applying Theorem \ref{main.1}.
\end{proof}

Since any morphism of commutative unital $\Bb$-monoids is easily seen to be central, a corollary follows immediately:

\begin{corollary}   Given a norm-decreasing homomorphism $\iota \colon \mathbf A \to \mathbf B$ of commutative unital Banach algebras,
the following conditions are equivalent:
\begin{itemize}
\item [(i)] $\iota$ is an effective descent morphism;
\item[(ii)] $\iota $ is a $\ot$--pure morphism in \emph{${{_\A}\!\Bb}$};
\item[(iii)] $\iota$ is a weak retract in \emph{${_{\A}\!\Bb}$}.
\end{itemize}
\end{corollary}

\begin{example}\em Let $c_0$ be the Banach space of all sequences $\textbf{a}=(a_n)_{n \in \,\mathbb{N}}$ of scalars
converging to zero with the supremum norm $\| \textbf{a} \|_\infty =\sup_{n \in \mathbb{N}}\{|a_n|\}$, $\ell^1$ the
space of all sequences for which the norm $\| \textbf{b} \|_1 =\sum_{n=1}^\infty |b_n|$ is finite, and $\ell^\infty$
the space of all bounded sequences of scalars with the some supremum norm as $c_0$.
Then $(c_0)^*$ is isometrically isomorphic to $\ell^1$ and $\ell^\infty$ to $(\ell^1)^*$ (e.g., \cite{P}).
With these isometrical isomorphisms, the canonical isometric inclusion of $c_0$ into its double dual can be identified
with the usual inclusion $c_0 \to \ell^\infty$ of spaces of sequences. Since both of $c_0$ and $ \ell^\infty$ with
element-wise algebra operations are commutative unital Banach algebras, it follows from Theorem \ref{main.2} that
the canonical inclusion $c_0 \to \ell^\infty$ of unital Banach algebras is an effective descent morphism.
\end{example}

We conclude the note by giving a result which shows how to construct an effective descent morphism for Banach modules
from any commutative unital Banach algebra.

Let $\A$ be an arbitrary unital Banach algebra. Then the second dual $A^{**}$  of $A$
can be equipped with two Banach algebra products, called \emph{first and second Arens products},
each of which makes it into a unital Banach algebra such that the canonical embedding $\iota_\A: A \to A^{**}$
is a homomorphism of unital  Banach algebra (e.g., \cite{P}). Since $\iota_\A$ is always a weak retract in ${_{\A}\!\Bb}$
(\cite{BP}), and since for commutative $\A$, $\iota_\A$ is central with respect to either Arens product
(see, e.g., \cite[3.1.14]{P}($c\,^\prime$)), Theorem \ref{main.2} gives:

\begin{proposition} Let $\A$ be any commutative Banach algebra. When $A^{**}$ is provided with either Arens product,
$\iota_\A: A \to A^{**}$ is an effective descent morphism.
\end{proposition}

\end{document}